\documentclass[11pt]{amsart} 
\usepackage{amsmath,amsfonts,amssymb}
\usepackage{amsthm}
\usepackage{xfrac}
\usepackage{hyperref}

\usepackage{color}

\theoremstyle{plain}
\newtheorem{thm}{Theorem}
\newtheorem*{thm*}{Theorem}
\newtheorem{lemma}[thm]{Lemma}
\newtheorem{prop}[thm]{Proposition}

\newtheorem{cor}[thm]{Corollary}

\newtheorem*{claim*}{Claim}

\newcommand{\fakeenv}{} 

\newenvironment{restate}[2]  
{
  \renewcommand{\fakeenv}{#2} 
  \theoremstyle{plain}
  \newtheorem*{\fakeenv}{#1~\ref{#2}} 
  \begin{\fakeenv}
}
{
  \end{\fakeenv}
}

\theoremstyle{definition}
\newtheorem{defn}[thm]{Definition}

\newtheorem{eg}[thm]{Example}
\newtheorem{rmk}[thm]{Remark}

\numberwithin{equation}{section}
\numberwithin{thm}{section}

\usepackage[utf8]{inputenc} 

\usepackage{geometry} 
\usepackage{graphicx} 

\title{\em Hyperbolic Immersions of Free Groups}
\author{Jean Pierre Mutanguha}

\address{Department of Mathematical Sciences \\
  University of Arkansas\\
  Fayetteville, AR
  \newline \indent  \it Web address: \tt {\url{https://mutanguha.com}}}
\email{\href{mailto:jpmutang@uark.edu}{jpmutang@uark.edu}}

\date{} 

\begin{document}

\begin{abstract}We prove that the mapping torus of a graph immersion has a word-hyperbolic fundamental group if and only if the corresponding endomorphism does not produce Baumslag-Solitar subgroups. Due to a result by Reynolds, this theorem applies to all injective endomorphisms of $F_2$ and nonsurjective fully irreducible endomorphisms of $F_n$. We also give a framework for extending the theorem to all injective endomorphisms of~$F_n$.
\end{abstract}
\maketitle

\section{Introduction}



 Thurston \cite{ThII} proved that the interior of a {\it mapping torus} $M_f$ of a hyperbolic surface homeomorphism $f:S \to S$ has a finite volume hyperbolic structure if and only if the homeomorphism is isotopic to a {\it pseudo-Anosov} homeomorphism, which by Nielsen-Thurston Classification is equivalent to saying $f$ has no periodic homotopy classes of essential simple closed curves. Assuming $S$ is closed, Thurston's result states, in particular, that $f$ is isotopic to a pseudo-Anosov homeomorphism if and only if $\pi_1(M_f)$ is {\it word-hyperbolic}, i.e., its Cayley complex satisfies a linear isoperimetric inequality \cite{ABC, Gro}. 

In the spirit of Thurston's result, Bestvina-Feighn proved that when $G$ is word-hyperbolic and $\phi:G \to G$ is a {\it hyperbolic} automorphism, then $G \rtimes_\phi \mathbb Z$ is word-hyperbolic. Hyperbolic automorphisms are defined in Section \ref{hyp}.
Peter Brinkmann later proved that {\it atoroidal} automorphisms of free groups, i.e., automorphisms with no nontrivial periodic conjugacy classes, are hyperbolic. Together, these theorems give:

\begin{thm*}[Bestvina-Feighn~\cite{BF92}, Brinkmann~\cite{Bri}] Let $\phi:F \to F$ be an automorphism of a free group of finite rank. The following are equivalent:
\begin{enumerate}
\item $\phi$ is hyperbolic.
\item $F \rtimes_\phi \mathbb Z$ is word-hyperbolic.
\item $F \rtimes_\phi \mathbb Z$ has no $\mathbb Z^2$ subgroup.
\item $\phi$ is atoroidal.
\end{enumerate}\end{thm*}

One might now ask if a similar statement is true when $\phi$ is injective but not surjective. In this case, we can no longer form a semi-direct product $F \rtimes \mathbb Z$ but the same presentation defines an {\it ascending HNN-extension} denoted by $F*_\phi$ (Section~\ref{defs}). Unlike automorphisms, an iterate of $\phi$ may now map a nontrivial element to a conjugate of some proper power.

\begin{eg}\label{a0} Let $F_1 \cong \mathbb Z$ be generated by $c$ and let $d$ be any positive integer, then we can define an endomorphism $\psi:F_1 \to F_1$ by $\psi(c) = c^d$. We denote $\mathbb Z*_{\psi}$ by $BS(1,d)$ and these are the {\bf metabelian Baumslag-Solitar (BS)} groups. For the rest of this paper, we will explore $BS(1,2)$.\end{eg}

Baumslag-Solitar groups are obstructions to word-hyperbolicity and our main theorem states that, under certain hypothesis, these are the only obstructions. 

\begin{restate}{Theorem}{mainthm} Suppose $f:\Gamma \to \Gamma$ is an immersion of a finite graph. The following are equivalent:
 \begin{enumerate} 
 \item $\pi_1(M_f)$ is word-hyperbolic.
 \item $\pi_1(M_f)$ contains no $BS(1,d)$ subgroups for $d \ge 1$.
 \item There are no $k, d \ge 1$ and nontrivial loop $\sigma$ in $\Gamma$ such that $f^k(\sigma)$ is freely homotopic to $\sigma^d$.
 \end{enumerate} 
 \end{restate}



While proving this theorem, we will derive one more equivalent condition that we currently omit until relevant definitions are given in Section \ref{pulls}. 
Our theorem generalizes theorems by Ilya Kapovich \cite[Theorem 5.5]{Kap} and Fran\c{c}ois Gautero \cite[Theorem~13.2]{Gau}: Kapovich assumed $\phi(F)$ was an {\it immersed subgroup} (equivalently, $f$ is an immersion on the rose) and the proof was algebraic; Gautero assumed $\phi$ was {\it hyperbolic} with malnormal image and their argument was topological. We give a topological proof along the lines of Kapovich which allowed us to generalize both results. The extra condition mentioned before also gives us an algorithm that determines whether Condition (3) holds for a given immersion (Remark \ref{algo}, Corollary~\ref{algo2}). Even though the hypothesis of Theorem \ref{mainthm} is restrictive, using a theorem of Patrick Reynolds~\cite{Rey}, we show that it applies to all {\it fully irreducible} endomorphisms of $F_n$ and all injective endomorphisms of $F_2$:
 \begin{restate}{Corollary}{cor2} Let $\phi:F_n \to F_n$ be a fully irreducible endomorphism. The following are equivalent:
\begin{enumerate}
\item $F_n*_\phi$ is word-hyperbolic. 
\item $F_n*_\phi$ has no $BS(1,d)$ subgroups for $d \ge 1$.
\item There are no $k, d \ge 1$ and $1 \ne g \in F_n$ such that $\phi^k(g)$ is conjugate to $g^d$ in $F_n$.
\end{enumerate}\end{restate}
 \begin{restate}{Corollary}{cor3} Let $\phi:F_2\to F_2$ be an injective endomorphism. The following are equivalent:
\begin{enumerate}
\item $F_2*_\phi$ is word-hyperbolic. 
\item $F_2*_\phi$ has no $BS(1,d)$ subgroups for $d \ge 1$.
\item There are no $k, d \ge 1$ and $1 \ne g \in F_2$ such that $\phi^k(g)$ is conjugate to $g^d$ in $F_n$.
\end{enumerate}\end{restate}

\begin{eg}[Sapir Group]\label{b0} Let $F_2$ be a free group generated by $a,b$ and $\varphi:F_2 \to F_2$ be given by $\varphi(a) = ab$ and $\varphi(b) = ba$. The {\bf Sapir Group} is the asc.~HNN-ext.~$F_2*_{\varphi}$. Since $\varphi(F_2)$ is an immersed subgroup, one can use Kapovich's result to prove the group is word-hyperbolic. The only proof we have found is due to J. O. Button \cite[Theorem 4.1]{But} who used the action on cyclically reduced words to directly show that no iterate of $\varphi$ maps a nontrivial element to a conjugate of its power. We will provide a topological proof of this fact. (Example \ref{b3}) \end{eg}

The two examples given in this section, $\psi$ and $\varphi$, will be used throughout the paper to illustrate the various ideas involved and three new examples will be given in Section~\ref{irred}. The reader is encouraged to choose a random nonsurjective injective endomorphism of $F_3$ to work with as another example. If the endomorphism is not induced by any graph immersion, this would be a useful example to have when generalizing Corollary~\ref{cor3} to $F_n$.

{~}

We now briefly sketch the proof of Theorem \ref{mainthm}. The first implication, $(1){\implies}(2)$, is the fact that BS subgroups are obstructions to word-hyperbolicity: word-hyperbolic groups have virtually cyclic centralizers and their cyclic subgroups are quasi-isometrically embedded \cite{ABC}. Kapovich proved $(2){\implies}(3)$ \cite[Lemma 2.3]{Kap}.  We prove $(3){\implies}(1)$ using the Bestvina-Feighn combination theorem. Briefly, the combination theorem states that if all annuli $\alpha: S^1 \times I \to M_f$ (with some technical conditions) have {\it uniform exponential growth}, then $\pi_1(M_f)$ is word-hyperbolic. It remains to show that all annuli do have uniform exponential growth. 

Informally, the growth of an annulus $g(\alpha)$ is the ratio of the larger end of the annulus to the center of the annulus and the length $l(\alpha)$ is the distance between the ends. The annuli have uniform exponential growth if there is $\lambda>1$ such that $\lambda^{l(\alpha)} \le g(\alpha)$ for long enough annuli.

Lift $\alpha$ to the natural infinite cyclic cover of $M_f$ to get $\tilde \alpha:S^1 \times I \to \tilde M_\mathbb{Z}$. This cover has a preferred direction in which things grow exponentially due to the atoroidal and immersion assumption on $f$. In particular, if $\tilde \alpha$ is monotone with respect to this direction, then it will have uniform exponential growth along this direction. So we need to show that those that fail to be monotone have uniform exponential growth as well. The definition of annuli forces such $\tilde \alpha$ to consist of two segments: one increasing, one decreasing. 

Proposition \ref{main2} uses the lack of $f$-invariant loops and immersion assumption on $f$ to show that, for all non-monotone $\tilde \alpha$, we can assume that the decreasing segments have uniformly bounded lengths. So long enough annuli will  have negligible decreasing segments and behave like monotone annuli, i.e., they have uniform exponential growth.  This concludes the sketch proof. Note that when $\phi$ is an automorphism, all annuli will be monotone but the uniform exponential growth is difficult to establish since $f$ is not an immersion. The second case where an annulus fails to be monotone is a new phenomenon unique to nonsurjective endomorphisms.

{~}

To generalize Theorem \ref{mainthm} to all $\pi_1$-injective graph maps, it would be helpful to have an analogue for hyperbolic automorphisms. In Section \ref{hyp}, we define {\it hyperbolic} graph maps and give a sufficient condition for when the fundamental groups of their mapping tori are word-hyperbolic. This reduces the general problem to showing an analogue of Brinkmann's theorem and generalizing Proposition~\ref{main}.

{~}

\noindent \textbf{Overview of the paper:} For the most part, we follow the structure of Kapovich's paper \cite{Kap}. In Section~\ref{defs}, we set the assumptions, definitions, and notations that will be used throughout the paper. In Section~\ref{pulls}, we introduce the pullback of a graph immersion and we prove a pullback stabilizing proposition (Proposition~\ref{main}). 
In Section~\ref{comb}, annuli and other relevant definitions are given and the combination theorem is stated. Section~\ref{anns} is the crucial bridge between the previous two sections as we interpret Proposition~\ref{main} in terms of annuli (Proposition~\ref{main2}). In Section~\ref{hyp}, we use this annuli interpretation and the combination theorem to prove word-hyperbolicity. In the process, we define hyperbolic graph maps and mention how this may help extend the main theorem from immersions to all $\pi_1$-injective graph maps. Section~\ref{irred} contains the application of the main theorem to fully irreducible endomorphisms of $F_n$ and all injective endomorphisms of $F_2$.

{~}

\noindent \textbf{Acknowledgments:} I would like to thank my advisor Matt Clay for the discussions that led me to this question and result. I am also grateful for the comments from Ilya Kapovich and their suggestion of Corollary \ref{cor3}, and Derrick Wigglesworth for their suggestion of Definition \ref{defhyp}. Finally, I thank the referee for suggestions and comments that helped me streamline the exposition.

\section{Definitions and Notations}\label{defs}

We define finite graphs to be finite $1$-dimensional CW-complexes. The $0$-cells are called the vertices and the $1$-cells are edges. A {\bf core graph} is a finite graph with no valence-$1$ vertices. In this section and the remainder of the paper, $\Gamma$ is a connected core graph whose vertices have valency~$\ge 3$. 

A (topological) map $f:\Gamma \to \Gamma$ is an {\bf immersion} if it is a locally injective map. The valency restriction on $\Gamma$ is only included as it implies that immersions map vertices to vertices. The restriction can be removed if, additionally, immersions are assumed to map vertices to vertices. The {\bf mapping torus of $f$}, $M_f$, is defined to be \[ \left( \Gamma \times \left[-\sfrac{1}{2}, \sfrac{1}{2}\right] \right)/{\sim}  \text{ with the relation } (x,\sfrac{1}{2}) \sim (f(x),-\sfrac{1}{2}) ~ \forall x \in \Gamma \] 

We also set, for this section and the remainder of the paper, $S^1 = \mathbb R / \mathbb Z$ and, for any interval $I \subset \mathbb R$, $I_{\mathbb Z} = I \cap \mathbb Z$. The {\bf edge space} is the integer cross-section $\Gamma \times \{ 0 \}$ while the {\bf vertex space} is the complement in $M_f$ of the edge space. The interval used in our definition for $M_f$ is not standard but it has been chosen so that the edge space, which is the space we will be most interested in, lies in the integer cross-section of $M_f$. This was a purely aesthetic choice.

An immersion $f:\Gamma \to \Gamma$ induces an injective  (outer) endomorphism $\phi = f_*$ of $F = \pi_1(\Gamma)$, well-defined up to post-composition with an inner automorphism. By Van Kampen's theorem:
\[ \pi_1(M_f) \cong \langle ~ F, t ~|~ t^{-1}xt = \phi(x), \forall x\in F ~\rangle \]
When $\phi$ is an automorphism, the latter is the presentation of the semi-direct product $F \rtimes_\phi \mathbb Z$. For an injective (not necessarily surjective) endomorphism $\phi$, it is the presentation of its {\bf ascending HNN-extension} and denoted by $F*_{\phi}$.

A map $f:\Gamma \to \Gamma$ has {\bf an invariant loop} if there exists positive integers $k,d$, and a nontrivial loop $\sigma$ in $\Gamma$ such that \( f^k(\sigma) \simeq \sigma^d\), i.e., $f^k(\sigma)$ is freely homotopic to $\sigma^d$. Equivalently, for the induced endomorphism $\phi = f_*$, there are $k,d \ge 1$ and a nontrivial $g \in F$ such that $[\phi^k(g)] = [g^d]$, i.e., $\phi^k(g) = x g^d x^{-1}$ for some $x \in F$. We shall refer to $d$ as the {\bf degree of the invariant loop}.

 
\begin{eg}[Continuing Example \ref{a0}]\label{a1}Let $G = S^1$. We can induce $\psi$ with the map $g:G \to G$ given by $x \mapsto 2x$. This is an immersion on the circle. As mentioned earlier, $\pi_1(M_{g}) \cong BS(1,2)$ is not word-hyperbolic. Clearly, the graph $G$ is a $g$-invariant loop with degree $d=2$.\end{eg}

\begin{eg}[Continuing Example \ref{b0}]\label{b1}Let $H$ be two copies of $S^1$ with their basepoints $0\in S^1$ identified and label the copies $a$ and $b$ respectively. Then we can induce $\varphi$ with an immersion $h: H \to H$ that maps $a$ onto the path $a b$ and $b$ onto $b a$. We shall eventually show that $\pi_1(M_h)$ is word-hyperbolic and $h$ has no invariant loop.\end{eg}

\section{Pullbacks}\label{pulls}

\begin{defn} The {\bf pullback} or {\bf fibered product} of graph immersions $g:A \to G$ and $h:B \to G$ is the topological space:
\[ A \times_G B = \left\{ ~ (x,y) \in A \times B ~ : ~ g(x) = h(y) ~\right\} \]
For $i \ge 1$, let $\Gamma_i$ be the pullback of $f^i$ and $f^i$. Set $\Gamma_0 = \left\{ (x,y) \in \Gamma \times \Gamma~:~x=y \right\}$.
\end{defn}

It follows from the local injectivity of an immersion $f:\Gamma \to \Gamma$ that each $\Gamma_i$ is a finite graph and $\Gamma \cong \Gamma_0 \subset \Gamma_1 \subset \Gamma_2  \subset \cdots \subset \Gamma \times \Gamma$.

\begin{defn}For any graph $X$, a {\bf direction at a point} $x \in X$ is an end of $X - x$. We abuse notation and denote the set of directions at $x$ by $T_xX$. Given a topological map $q:X \to X$ that preserves vertices and is locally injective on interior of edges, the {\bf derivative map} of $q$ at $x$ is the induced map $dq_x:T_xX \to T_{q(x)}X$.\end{defn}

Note that the valency at $x$ is the number of directions at $x$ and all derivative maps are injective if and only if the graph map is an immersion. The definition is given here rather than the previous section since it will only be used in the following lemma and Lemma~\ref{explem}.

\begin{lemma}\label{lem0} Let $f:\Gamma \to \Gamma$ be an immersion. Then for all $i \ge 0$, $\Gamma_i$ is a union of some components of $\Gamma_{i+1}$.\end{lemma}
\begin{proof} Fix some $i \ge 0$ and let $C$ and $C'$ be components of $\Gamma_i$ and $\Gamma_{i+1}$ resp. such that $C \subset C'$. We want to show that $C = C'$. 

For any $p \in C$, let $v_p = C$-valency of $p$, i.e., the number of ends of $C - p$, and similarly define $v'_p = C'$-valency of $p$. Recall $p$ corresponds to two points $p_1, p_2 \in \Gamma$ such that $f^i(p_1) = f^i(p_2)$ and let $q=f^i(p_1)$. The $C$-valency of $p$ corresponds to a maximal list of pairs of directions $(a_1, \alpha_1), \ldots, (a_{v_p},\alpha_{v_p})$ based at $p_1$ and $p_2$ resp. such that $df^i_{p_1}(a_j)=df^i_{p_2}(\alpha_j)$. Suppose $v_p < v'_p$, then that means there is at least one extra pair of directions $(b, \beta)$ such that $df^i_{p_1}(b) \neq df^i_{p_2}(\beta)$ but $df^{i+1}_{p_1}(b) = df^{i+1}_{p_2}(\beta)$. But this means that $q = f^i(p_1) = f^i(p_2)$ has two distinct directions $df^i_{p_1}(b), df^i_{p_2}(\beta)$ that map to the same direction under the derivative map $df_q$. This contradicts the fact $f$ is an immersion. Therefore, $v_p = v'_p$ and thus $C = C'$ as all points of $C$ have the same valency in both $C$ and $C'$.
\end{proof}

For this section and the remaining sections, set $\hat \Gamma_{i}$ to be the maximal core subgraph of $\Gamma_{i}-\Gamma_{i-1}$. There is a natural immersion $\hat f: \Gamma_{i} \to \Gamma_{i-1}$ given by $\hat f(x,y) = (f(x), f(y))$ which restricts to an immersion $\hat \Gamma_i \to \hat \Gamma_{i-1}$. 

\begin{lemma}\label{lem1} Let $f:\Gamma \to \Gamma$ be an immersion. If $\hat \Gamma_i$ is empty, then so are all $\hat \Gamma_j$ for $j > i$. If $\hat \Gamma_i$ consists of loops and $\hat \Gamma_{i+1}$ is nonempty, then $\hat \Gamma_{i+1}$ consists of loops too.\end{lemma}
\begin{proof} Since we have the immersion $\hat f: \hat \Gamma_i \to \hat \Gamma_{i-1}$ for all $i \ge 1$, if $\hat \Gamma_i$ is empty, then so are all $\hat \Gamma_j$ for $j > i$. Furthermore, the only core graph that immerses into a loop is a loop itself. So if $\hat \Gamma_i$ consists of disjoint loops and $\hat \Gamma_{i+1}$ is nonempty, then $\hat \Gamma_{i+1}$ consists of disjoint loops. 
\end{proof}

We say that the {\bf pullbacks stabilize} if $\hat \Gamma_i = \emptyset$ for some $i$.

 \begin{eg}[Continuing Example \ref{a1}]\label{a2} For $g:G \to G$, the pullback $G_1$ is two disjoint loops: one is the diagonal $G_0 \cong G$ and the other is $ \hat G_1 = \left\{ (x,y) \in S^1 \times S^1 :  y - x = \sfrac{1}{2} \right\}$. More generally, \[ \hat G_i =\left\{ (x,y) \in S^1 \times S^1 : y - x = \frac{(2j-1)}{2^i}, ~ j\in [1, 2^{i-1}]_\mathbb{Z} \right\} \]
 Topologically, the components for all $G_i$ are loops and the number of components in $\hat G_i$ doubles with each iteration. The picture on the left in Figure~\ref{fig1} shows the first two pullbacks.
 \end{eg}
 
 \begin{eg}[Continuing Example \ref{b1}]\label{b2} For $h:H \to H$, the pullback $H_1$ consists of a copy of $H$ and two extra loops. The second pullback $H_2$ consists of $H_1$ and contractible components. Therefore,  the core subgraph $\hat H_2$ and all subsequent $\hat H_i$ are empty and so the pullbacks stabilize. Contrast this behavior with that of $g$ in previous example. Proposition~\ref{main} below states that pullbacks stabilize if the immersion has no invariant loops. The picture on the right in Figure~\ref{fig1} illustrates $H_1$.
 \end{eg}
 
\begin{figure}[h]
 \centering 
 \includegraphics[scale=0.40]{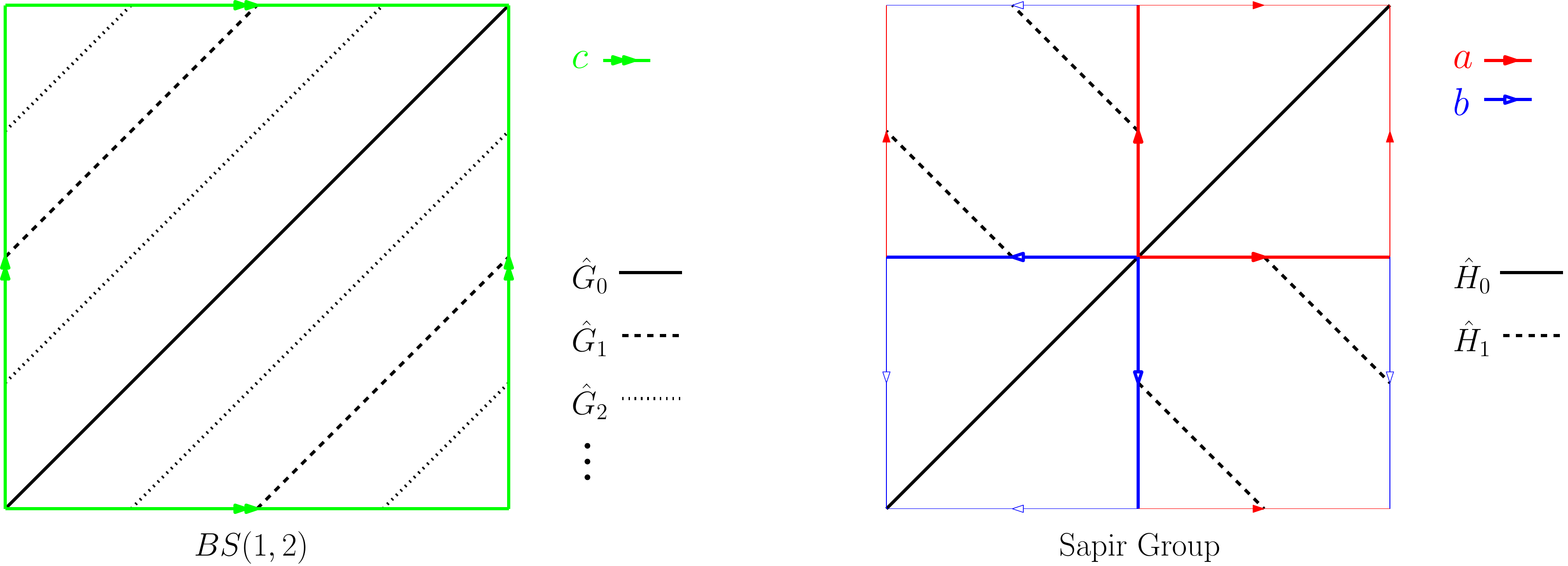}
 \caption{Pullbacks for our main examples.}
 \label{fig1}
\end{figure}



Walter Neumann used pullbacks to generalize a result by Hanna Neumann \cite{HN}:


\begin{thm}[{\cite[Proposition~2.1]{WN}}]\label{shnc} Let $M,N \le F$ be nontrivial finitely generated subgroups. Then 
\[ \sum_{[[g]] \in M \backslash F /N} \max(0, \mathrm{rank}(M \cap gNg^{-1})-1) \le 2 (\mathrm{rank}(M)-1)( \mathrm{rank}(N)-1) \]
where the sum ranges over all $(M,N)$-double cosets $[[g]] = MgN$. \end{thm}

\begin{rmk} The {\it strengthened Hanna Neumann conjecture} is that the $2$ can be dropped and it has been proven (Friedman~\cite{Fried}, Mineyev~\cite{Min}) but, for our purporses, Walter Neumann's bound will suffice. \end{rmk}

\begin{lemma}\label{lem2}Let $f:\Gamma \to \Gamma$ be an immersion and $M=2(\mathrm{rank}(\Gamma)-1)^2$. For $i \ge M$, the graphs $\hat \Gamma_{i}$ are either all disjoint loops or eventually empty. \end{lemma}
\begin{proof}
Let $\phi = f_*:F \to F$ be the induced endomorphisms. The components of $\Gamma_{i}$ correspond to the intersections $\phi^i(F) \cap g\phi^i(F)g^{-1}$ as $[[g]]$ ranges over $\phi^i(F)$-double cosets. Hence, we have the following inequality for all $i$: \[ \chi^-(\Gamma_i) = \sum_{[[g]]} \max(0, \mathrm{rank}(\phi^i(F) \cap g\phi^i(F)g^{-1})-1)  \le 2(\mathrm{rank}(\phi^i(F))-1)^2 = M \]
where $\chi^-$ is the negative Euler characteristic of the maximal core subgraph and the sum ranges over all $\phi^i(F)$-double cosets $[[g]]$.
The chain $\Gamma_0 \subset \Gamma_1 \subset \cdots$ and Lemma~\ref{lem1} imply that the nondecreasing sequence $(\chi^-(\Gamma_i))_{i=0}^\infty$ becomes constant once two consecutive terms are equal. Since the sequence is bounded by $M$, the subsequence $(\chi^-(\Gamma_i))_{i=M}^\infty$ must be constant and, by Lemma \ref{lem1} again, the graphs $\hat \Gamma_{i}$ are disjoint loops for all $i \ge M$ or are eventually empty. 
\end{proof}

\begin{eg} For the immersion $g$ in Example \ref{a2}, we get $M_1 = 2(1-1)^2 = 0$. So we conclude from the bound that the graphs $\hat G_i ~ (i \ge 0)$ are either all disjoint loops or eventually empty. As we verified earlier, it is the former case. 

In Example \ref{b2}, we have $M_2 = 2(2-1)^2 = 2$. So $\hat H_i ~ (i \ge 2)$ are either all disjoint loops or eventually empty.  We verified earlier that it is the latter case.\end{eg}

\begin{prop}\label{main} Suppose $f:\Gamma \to \Gamma$ is an immersion and $\hat \Gamma_i$ is nonempty for all $i$. Then $f$ has an invariant loop with degree $d \ge 2$.\end{prop}
\begin{proof} Suppose $\hat \Gamma_i$ is nonempty for all $i$. Let $( \gamma_k)_{k=0}^\infty$ and $(m_k)_{k=0}^\infty$ be sequences of immersed loops  and strictly increasing positive integers such that $\gamma_k \subset \hat \Gamma_{m_k}$. Lemma \ref{lem2} gives an explicit $M$ such that, for all $m \ge M$, the components of $\hat \Gamma_m$ are loops. Passing to subsequence if necessary, assume that $m_k > M$ for all $k$ and the sequence of loops $(\hat f^{m_k - M}(\gamma_{k}))$ in $\hat \Gamma_M$ all lie in the same component, $C$; this can be done as $\hat \Gamma_M$ is finite. Therefore, for all $k$, there is a nonzero integer $n_k$ satisfying $\hat f^{m_k - M}(\gamma_k) \simeq C^{n_k}$. The component $C$ is a pair of loops $(c_{-}, c_{+}) \subset \Gamma \times \Gamma$ and, similarly, $\gamma_k$ is a pair of loops $(\sigma_{-k}, \sigma_k)$ so that we have $f^{m_k - M}(\sigma_{\pm k}) \simeq c_{\pm}^{n_k}$. We are using $\pm$ signs to distinguish the left and right factors.

Let $Z \subset \Gamma$ denote the subgraph of non-expanding edges, i.e., edges whose images under $f$-iteration have uniformly bounded combinatorial lengths. Then $f$-orbits of edges of $Z$ are eventually periodic. In particular, for large enough $k$, if $x,y \in Z$ and $f^k(x) = f^k(y)$ then $f^{k-1}(x) = f^{k-1}(y)$. So for large enough $k$, $(x,y) \in \hat\Gamma_k$ implies at least one of $x$ or $y$ is not in $Z$, i.e., at least one of $x$ or $y$ is a point in an expanding edge. By passing to a subsequence if necessary, we can assume without loss of generality that all loops $\sigma_{k}$ share an expanding edge $e \subset \Gamma$. 

Set $c$ to be the minimal root of $c_{+}$ and $N = N(c)$ to be the number of subpaths of $c$ that are also loops (subloops of $c$).

Since $f$ is an immersion, $e$ is an expanding edge of $\sigma_k$ for all $k$, and $f^{m_k-M}(\sigma_{k}) \simeq c_{+}^{n_k}$, we have $f^i(e)$ surjects onto $c$ for large enough $i$. In fact, it will surject onto any power of $c$ for large enough $i$. Fix a large $k$ and a subsegment $s \subset e$ such that $f^{m_k-M}(s) = c^{N+1}$. Let $m=m_{k+1} - m_k$, then $f^{m_{k+1}-M}(s) = f^m(f^{m_k-M}(s)) = f^{m}(c^{N+1})$ is a subloop of $f^{m_{k+1}-M}(\sigma_{k+1}) \simeq c^{n_{k+1}}$. In fact, for all $1 \le j \le N+1$, $f^{m}(c^j)$ is a subloop of $c^{n_{k+1}}$. Thus, there is a subloop $\epsilon_j \subset c$ and a positive integer $s_j$ such that
\[ f^{m}(c^j) \epsilon_j \simeq c^{s_j} \]
By the choice of $N$, there is a pair $1 \le t < t' \le N+1$ such that $\epsilon_t = \epsilon_{t'} = \epsilon$. Therefore,
\[\begin{aligned} f^{m}(c^t) \epsilon &\simeq c^{s_t}, & f^{m}(c^{t'}) \epsilon &\simeq c^{s_{t'}} \end{aligned}\]
and \[ f^m(c)^{t'-t} = f^{m}(c^{t'-t}) \simeq c^{s_{t'}-s_t} \]
We also have the following inequalities: $m = m_{k+1}-m_k >0$, $t'-t>0$, and $ s_{t'} - s_t >0$. As $c$ is a minimal root, the last equation implies $f^{m}(c) \simeq c^d$ for some $d \ge 1$; in fact, $d \ge 2$ since $e \subset c$ is expanding.
\end{proof}

\begin{rmk}\label{algo} Using the pigeonhole principle, there is a recipe for a number $L(f)$ such that if the original sequence $(m_k)$ had $L(f)$ entries instead of being infinite, then the necessary subsequences can still be taken producing a nonempty subsequence. In particular,  \( \hat \Gamma_{L(f)} \neq \emptyset \text{ implies } f \text{ has invariant loop with degree } d\ge 2. \) This recipe uses the following computable numbers: $M = 2(\mathrm{rank}(\Gamma)-1)^2$; the number of components of $\hat \Gamma_M$; the number of edges in $\Gamma$; the maximum $N(c_{-})$ or $N(c_{+})$ as $C = (c_{-}, c_{+})$ ranges over all components of $\hat \Gamma_M$; and the maximum combinatorial length of $c_{-}$ or $c_{+}$ as $C$ ranges over all components of $\hat \Gamma_M$. \end{rmk}

\section{The Combination Theorem}\label{comb}

The Bestvina-Feighn combination theorem gives us a sufficient condition for when mapping tori of free groups, and more generally graphs of $\delta$-hyperbolic spaces, have word-hyperbolic fundamental groups. We need to define some terms before stating it.

\begin{defn}Let $m$ be a positive integer. An {\bf annulus of length $2m$} is a map $\alpha: S^1 \times [-m,m] \to M_f$ satisfying the following conditions:
\begin{enumerate}
\item It is transverse to the edge space.
\item The $\alpha$-preimage of the edge space is $S^1 \times [-m,m]_\mathbb{Z}$.
\item For $i \in  [-m, m]_{\mathbb Z}$, $\alpha|_{S^1 \times \{ i\}}$ is locally injective everywhere with exception possibly at the basepoint $0 \in S^1 \cong S^1 \times \{i\}$.
\item \label{c4} For $i \in   [-m, m{-}1]_{\mathbb Z}$, $\alpha|_{\{ 0 \} \times [i, i+1]}$ is not homotopic rel endpoints into  the edge space
\end{enumerate}
For simplicity, set $\alpha_t = \alpha|_{S^1 \times \{ t \}}{:}S^1 \to M_f$ and $\alpha^* = \alpha|_{\{0\} \times [-m , m]}{:}[-m,m] \to M_f$.
\end{defn}

When $t$ is an integer, we will refer to the based loops $\alpha_t$ as the {\bf rings} of the annulus; the path $\alpha^*$ will be {\bf trace} of the basepoint. Informally, the rings will be used to define the thickness of the annulus while the pieces of the trace in the vertex space will define the distance between consecutive rings of the annulus.

{~}

\noindent The {\bf girth} of $\alpha$ is $l(\alpha_0)$ where $l$ is the (combinatorial) length measured after tightening rel. basepoint in the edge space. For $\lambda >1$, we say that $\alpha$ is {\bf $\lambda$-hyperbolic} if 
\[ \lambda l(\alpha_0) \le \max\{l(\alpha_{-m}), l(\alpha_m) \} \]
For every $i \in [-m,m{-}1]_\mathbb{Z}$, define $\tau_i:(i, i{+}1) \to \Gamma $ by \[ \tau_i(t) = 
\begin{cases} 
x & \text{ if } \alpha^*(t) = (x,s) \text{ and } 0 < s  \\
f(x) & \text{ if } \alpha^*(t) = (x,s) \text{ and } 0 > s 
\end{cases}\]
The path $\tau_i$ can be thought of as a projection of the piece $\alpha^*|_{(i,i+1)}$ to a cross-section of the vertex space between $\alpha_i$ and $\alpha_{i+1}$. The annulus $\alpha$ is {\bf $\rho$-thin} if $l(\tau_i) + 1 \le \rho$ for all $i \in [-m,m{-}1]_\mathbb{Z}$. Here, the length is measured after tightening  $\tau_i$ rel.~end points to be locally injective. This is akin to putting a taxicab-like metric on the universal cover to measure distances in the vertex space. A family of annuli that will be useful in the next section are the $1$-thin annuli. These are annuli whose trace of the basepoint projects to a null-homotopic path. One can also think of these annuli as homotopies of loops in the edge-space that respect the natural (semi)flow lines of $M_f$.

Now we can state the {\it combination theorem for mapping tori}:

\begin{thm}[{\cite[combination theorem]{BF92}}] Let $f:\Gamma \to \Gamma$ be a $\pi_1$-injective map. If there are numbers $\lambda > 1$, $m \ge 1$, and a function $H:\mathbb R \to \mathbb R$ such that any $\rho$-thin annulus of length $2m$ with girth at least $H(\rho)$ is $\lambda$-hyperbolic, then $\pi_1(M_f)$ is word-hyperbolic.\end{thm}

The hypothesis will be referred to as {\bf the annuli flare condition}. Annuli $\alpha$ lift to hallways in the universal cover, $\tilde \alpha: I \times [-m,m] \to \tilde M_f$. In this context, the number $\rho$ bounds the amount of shearing in the hallways and $1$-thin means there is no shearing. $H$ is a lower-bound for the thickness of the hallways' middle/girth. The flare condition is saying: Once we've bounded the shearing of the hallways, if we assume their girths are thick enough, then the hallways will get thicker towards one of their ends.

Technically, the combination theorem also requires that a {\it qi-embedded} condition is satisfied by the edge spaces, or equivalently, a quasi-convexity condition is satisfied by their fundamental groups. All our (vertex and edge) groups are finitely generated (f.g) free groups; f.g. subgroups of f.g. free groups are quasi-convex: free factors and finite index subgroups of f.g. free groups are quasi-convex and f.g. subgroups are free factors of finite index subgroups. So the condition is always satisfied under our assumptions. 

\begin{eg}[Continuing to Example \ref{a2}]\label{a3} We will exhibit two families of annuli: one that has $2^{i}$-hyperbolic members for all $i \ge 1$; another that has no member that is $\lambda$-hyperbolic for some $\lambda>1$. Let $c \subset G$ be the based loop defined earlier.

The first family: Fix $i\ge 1$, we have $g^{2i}(c) = c^{2^{2i}}$; therefore, the based loops $c$ and $c^{2^{2i}}$ in the edge-space of  $M_{g}$ are homotopic in $M_{g}$. Let $\alpha$ be some homotopy between the two based loops that is also a 1-thin annulus of length $2i$. Thus, $l(\alpha_0) = l(g^i(c)) = l(c^{2^i})=2^i$, $l(\alpha_i) = 2^{2i}$, and $2^{i} l(\alpha_0) \le l(\alpha_i)$. See the left half of Figure~\ref{fig2} for $i=2$.

The second family: Fix $i\ge 2$ and let $c'$ be the based loop given by $c'(x) = c(x+\sfrac{1}{2^i})$. Then $g^{i}(c')=c^{2^{i}}=g^{i}(c)$ as based loops; therefore, we get homotopies $c' \simeq g^{i-1}(c') \simeq c^{2^{i-1}} \simeq g(c)$ in $M_g$. Choose the homotopies so that their concatenation produces a $1$-thin annulus of length $2i-2$. Therefore, $l(\alpha_0) = l(g^{i-1}(c')) = 2^{i-1}$ but $l(c')=l(\alpha_{-i+1}) = 1$ and $l(\alpha_{i-1}) = l(g(c))= l(c^2) = 2$. These annuli are the antithesis of being $\lambda$-hyperbolic. See the right half of Figure~\ref{fig2} for $i=3$.
\end{eg}

\begin{figure}[h]
 \centering 
 \includegraphics[scale=0.37]{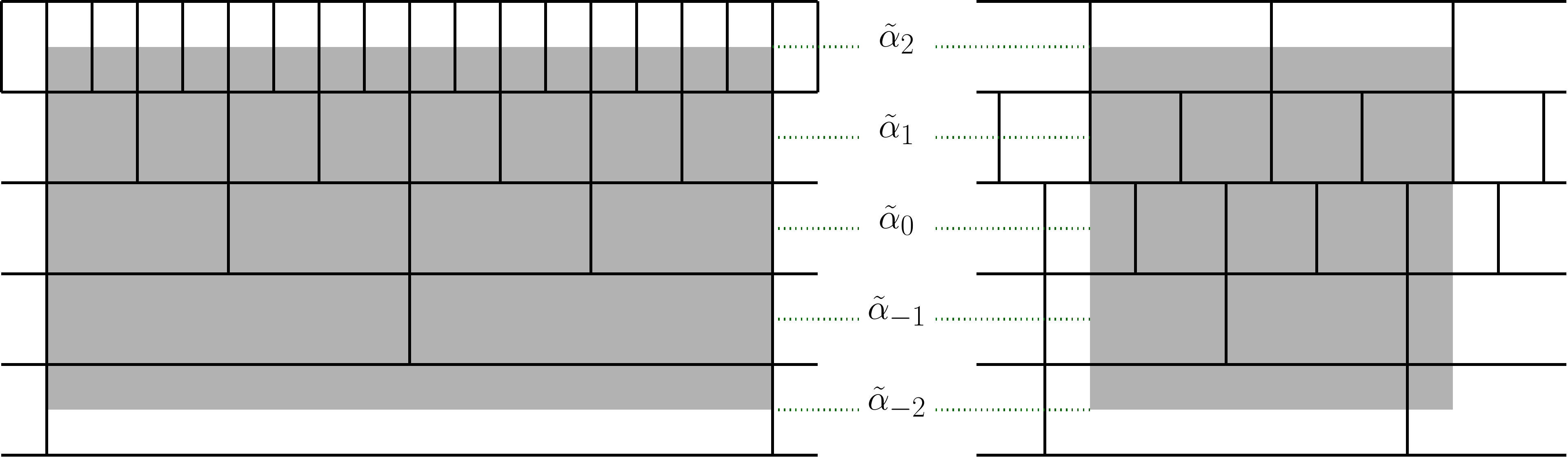}
 \caption{Lifts of two 1-thin annuli of length 4 to the universal cover $\tilde M_{g}$.}
 \label{fig2}
\end{figure}

\section{Annuli and Pullbacks}\label{anns}

Let $p : M_f \to S^1$ be given by $(x,t) \mapsto t$, $\alpha$ be an annulus of length $2m$, and 
$\beta:[-m,m] \to \mathbb R$, a lift of $p \circ \alpha^*$ to the universal cover of $S^1$.

\begin{defn} An annulus $\alpha$ is {\bf unidirectional} if $\beta$ is either increasing or decreasing on $[-m,m]$, equivalently, $\beta|_{[-m, m]_\mathbb{Z}}$ is strictly increasing or decreasing. Otherwise, it is {\bf bidirectional}. \end{defn}

Due to the fourth condition in the definition of an annulus, if $\alpha$ is bidirectional then $\beta$ switches directions exactly once and it switches from increasing to decreasing.

\begin{defn} An annulus $\alpha$ is {\bf strictly bidirectional} if $\beta(-m) = \beta(m) - 1 < \beta(0)$.  \end{defn}

We define a {\bf homotopy of annuli} $\alpha, \alpha'$ to be a homotopy that restricts to an edge space homotopy on each ring. Equivalently, a homotopy $\alpha \simeq \alpha'$ that induces a homotopy $\beta \simeq \beta'$ rel. integer points.  By this definition, being (strictly) bidirectional becomes a property of the homotopy classes of annuli. For simplicity's sake, we choose class representatives that are {\bf tight annuli}, i.e., the rings are locally injective even at the basepoint. For the rest of this section, all annuli and loops are assumed to be tight.

In the language of annuli, pullback stabilizing can be interpreted as:

\begin{prop}\label{main2}
Suppose $f:\Gamma\to \Gamma$ is an immersion. Then $\hat\Gamma_L$ is empty for some $L$ if and only if strictly bidirectional tight annuli have lengths uniformly bounded by some $2L$.
\end{prop}
\begin{proof}There is a correspondence between (classes of) strictly bidirectional annuli in $M_f$ and pullbacks of $\Gamma$. Let $\alpha$ be a strictly bidirectional annulus. For $i \in [-m, m]_\mathbb{Z}$, define:
\[ \alpha'_i = \begin{cases} f(\alpha_i) & \text{ if } -m \le i \le 0  \\ \alpha_i & \text{ otherwise.}\end{cases} \]
The $\alpha'_i$ have the nice property that $f^{|i|}(\alpha'_i) \simeq \alpha'_0$ in $\Gamma$.  We can and will henceforth choose a homotopy class representative $\alpha$ so that $f(\alpha'_i) = \alpha'_{i+1}$ if $i<0$, $f(\alpha'_i) = \alpha'_{i-1}$ if $i>0$, and $\tau_i$ is null-homotopic rel. endpoints for all $i$. The last condition is equivalent to saying $\alpha$ is $1$-thin.

Let $i > 0$. Since $f^i(\alpha'_{-i}) = \alpha'_0 = f^i(\alpha'_{i})$, by the definition of pullbacks, the pair $(\alpha'_{-i}, \alpha'_i)$ is a loop $\gamma_i$ in $\Gamma_i$. { Condition \ref{c4} in the definition of annuli and the assumption $\alpha$ is $1$-thin imply $\gamma_1 \subset \hat \Gamma_1$, and more generally, $\gamma_i \subset \hat \Gamma_i$ for all $i$.} 

Conversely, a loop $(\sigma_{-}, \sigma_{+}) \subset \hat \Gamma_m$ completely determines a $1$-thin strictly bidirectional annuli with length $2m-2$; length can be extended to $2m$ if $\sigma_{-}$ is the $f$-image of some loop in $\Gamma$. This in turn determines a homotopy class of annuli.  Thus, we can view classes of strictly bidirectional annuli of length $2m$ as certain loops in $\hat \Gamma_m$. \end{proof}

\begin{cor}\label{cor1} Suppose $f:\Gamma \to \Gamma$ is an immersion and $\hat \Gamma_L = \emptyset$ for some $L$. Then, up to reversal of direction, an annulus $\alpha$ with length greater than $4L$ will satisfy \[ \beta(-m) < \beta(0) < \beta(m)-1.\] \end{cor}

\begin{eg}[Continuing Example \ref{a3}] The proof of Proposition \ref{main2} gives a correspondence between classes of strictly bidirectional annuli and components of $\hat G_i$. Since the latter are never empty, we get that $M_{g}$ has strictly bidirectional annuli with unbounded lengths. Indeed, the second family of annuli we constructed in Example \ref{a3} are strictly bidirectional and unbounded; On the other hand, the first family are all unidirectional. \end{eg}

\begin{eg}[Continuing Example \ref{b2}] To contrast, we know that $\hat H_1$ consists of loops and $\hat H_2$ is empty. Since the definition of annulus requires them to have even length, $M_{h}$ has no bidirectional annuli. If we relax the definition of annuli to allow odd length, then all bidirectional annuli in $M_{h}$ have length $1$ and are homotopies $a^k \simeq b^k $ of edge-space loops.  Figure~\ref{fig3} is one explicit annulus with $k=2$. \end{eg}

\begin{figure}[h]
 \centering 
 \includegraphics[scale=0.57]{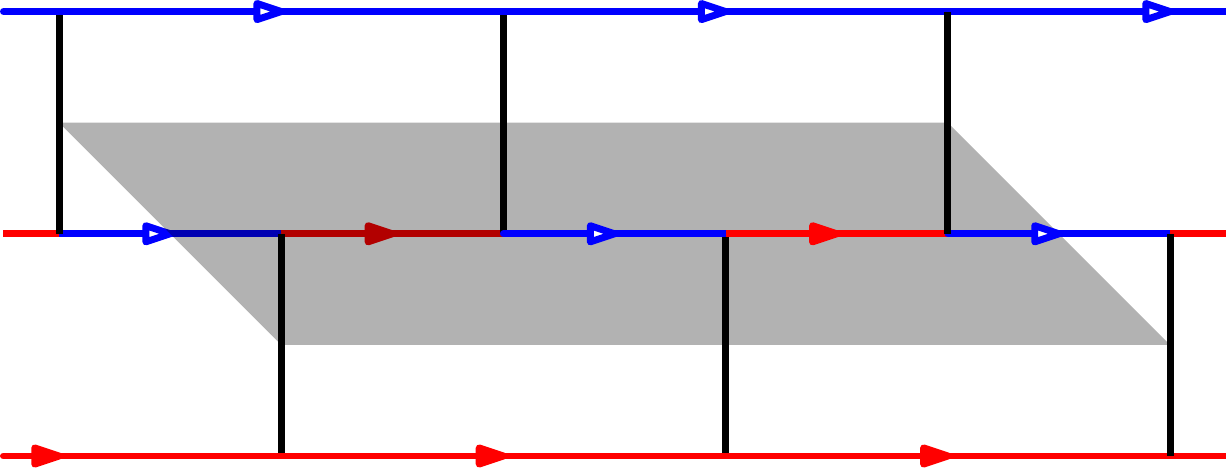}
 \caption{A lift of a $2$-thin bidirectional annulus with length $1$.}
 \label{fig3}
\end{figure}

\section{Hyperbolic Endomorphisms} \label{hyp}

An automorphism $\phi:G \to G$ of a word-hyperbolic group $G$ is said to be $\bf hyperbolic$ if:

\begin{center}
$\exists~\lambda>1$, $n\ge1$ such that for all nontrivial $g \in G, \lambda |g| \le \max( |\phi^n(g)|, |\phi^{-n}(g)| )$
\end{center} where $|\cdot|$ refers to the word-length with respect to some fixed finite generating set of $G$.

Bestvina-Feighn used their combination theorem to show that mapping tori of hyperbolic automorphisms are word-hyperbolic and we will prove an analogous statement for endomorphism of free groups (Theorem~\ref{keyprop}). First, we need to generalize the definition of hyperbolic automorphisms of free groups.

For the remainder of this section, assume a graph map $f:\Gamma \to \Gamma$ is a topological map that sends vertices to vertices and is locally injective on the interior of edges. In particular, $f$ maps edges to nontrivial edge-paths but is not necessarily locally injective at the vertices. Let $\lambda_f$ be the maximum $l(f(e))$ over all the edges $e$ in $\Gamma$, where $l(\cdot)$ is the combinatorial lengths of edge-paths in $\Gamma$. By our assumptions, $\lambda_f \ge 1$.

\begin{defn}\label{defhyp}For $\lambda>1$ and $n \ge 1$, we say a graph map $f:\Gamma \to \Gamma$ is {\bf $(\lambda,n)$-hyperbolic} if for all based nontrivial loops $\sigma: (S^1, *) \to \Gamma$, at least one of these two conditions holds:
\begin{enumerate}
\item \label{c1} $\lambda l(f^n(\sigma)) \le l(f^{2n}(\sigma))$
\item \label{c2} $\lambda l(f^n(\sigma)) \le l(\sigma)$
\end{enumerate}
We will say $(\lambda, n)$-Condition \ref{c1}/\ref{c2} holds for $\sigma$ if the corresponding condition above holds. Note that we assume the basepoint $\sigma(*)$ is a vertex of $\Gamma$, and the combinatorial lengths of based loops are measured after tightening rel. basepoint.
\end{defn}
\begin{rmk}\label{lambdafrmk} If $\lambda_f=1$, then $l(f^n(\sigma)) = l(\sigma)$ for all $n \ge 1$ and $f$ is not $(\lambda,n)$-hyperbolic since the definition of the latter requires $\lambda > 1$.
\end{rmk}

\begin{lemma}\label{explem} Suppose $f:\Gamma \to \Gamma$ is an immersion with no invariant loop of degree $d=1$, i.e., $f$ is atoroidal. Then $f$ is $(2,n)$-hyperbolic for some $n \ge 1$.\end{lemma}
\begin{proof} Suppose $\Gamma$ has no nontrivial $f$-invariant forest. If there exists an edge $e$ such that $l(f^n(e)) = 1$ for all $n \ge 1$. Then the $f$-orbit $\mathcal O = \cup_{n=1}^\infty f^n(e)$ is an $f$-invariant subgraph of $\Gamma$. By our assumption, $\mathcal O$ is not a forest, hence it contains an embedded loop $\sigma$. The orbit of $\sigma$ consist of loops with the same length and there are only finitely many of them. Thus the orbit is eventually periodic, which contradicts the atoroidal assumption. Therefore, for all edges $e$, there is an $n_e \ge 1$ such that $l(f^{n_e}(e)) \ge 2$. Set $n = \max\{n_e ~:~ e \text{ is an edge.} \}$. Thus $l(f^n(e)) \ge 2$ for all edges $e$ and $f$ is $(2,n)$-hyperbolic since $f$ is an immersion.

Suppose $\Gamma$ has a nontrivial $f$-invariant forest. Collapse a maximal nontrivial $f$-invariant forest to produce a graph $\Gamma'$ and map $f':\Gamma'\to\Gamma'$. Since $f$ is an immersion, it is a homeomorphism when restricted to a collapsed tree (component of forest). In particular, distinct boundary points of a component of the forest have distinct $f$-images. Suppose $v' \in \Gamma'$ is the image of a vertex in $\Gamma$ whose $f$-image is in a collapsed tree. The directions at $v'$, $T_{v'}\Gamma'$, correspond to directions at the boundary of a collapsed tree $T_{\partial Y} \Gamma$ if $v'$ itself is the image of a collapsed tree $Y$, otherwise they correspond to the directions of some vertex of $\Gamma$ outside the collapsed forest $F$. Since $f$ is injective on the boundary of a collapsed tree and $df$, the union of all derivative maps, is injective, then $df'_{v'}$ is also injective. Clearly, $df'_{v'}$ is also injective if $v'$ is the image of a vertex in $\Gamma$ whose $f$-image lies outside the collapsed forest. Therefore, $df'$ is injective and $f'$ is an immersion. By the previous paragraph, $f'$ is $(2,n')$-hyperbolic for some $n'$. Choose $k$ such that $2^k - 1 > $ number of edges in the maximal forest. If we let $n = kn'$, then for any immersed loop in $\Gamma$, any edge of the loop lying outside the maximal forest has $f^n$-image that is at least $2^k$ edges long, which will be longer than any immersed path lying in the forest. Thus $f$ is $(2,n)$-hyperbolic. \end{proof}

We can now state a generalization of Bestvina-Feighn's theorem on mapping tori of hyperbolic automorphisms: 

\begin{thm}\label{keyprop} If $f:\Gamma \to \Gamma$ is a $(\lambda,n)$-hyperbolic graph map and strictly bidirectional annuli in $M_f$ have lengths bounded by $2L$ for some $L > 1$, then $\pi_1(M_f)$ is word-hyperbolic.\end{thm}

\begin{rmk}Gautero proved a weaker version of this theorem  with the assumption that $f_*(\pi_1(\Gamma))\le\pi_1(\Gamma)$ is malnormal \cite[Theorem 13.2]{Gau}. This is equivalent to assuming $M_f$ has no bidirectional annuli. It is worth noting that Gautero does not use the combination theorem.\end{rmk}

\begin{proof} Choose $k \ge 1$ so that $\lambda^k > 2$ and we have $f$ is $(2, kn)$-hyperbolic. Fix $r \ge 1$ such that $2^r (\sfrac{2}{\lambda_f^{kn}})^L \ge 8$. Define $m = (2L+r)kn$ and 
\[ H(\rho) = \frac{4\rho}{\lambda_f-1}(\lambda_f^{2m}-1)2^m \] 
By Remark \ref{lambdafrmk}, $\lambda_f > 1$ as $f$ is $(\lambda,n)$-hyperbolic. Suppose $\alpha$ is a $\rho$-thin annulus of length $2m$ with $l(\alpha_0) \ge H(\rho)$. We will show that $\alpha$ is $7$-hyperbolic. Without loss of generality, assume $\beta(-m) = \beta(0)-m$. The lower half of the annulus gives us the following homotopy in $\Gamma$ rel. basepoint:

\[ \label{low} f^m(\alpha_{-m}) \simeq f^{m-1}(\tau_{-m})\cdots \tau_{-1} \alpha_0 \bar \tau_{-1} \cdots f^{m-1}(\bar \tau_{-m}) \]
This homotopy allows us to bound the length of $\alpha_0$ in terms of $f^m(\alpha_{-m})$:
\[ \begin{aligned}\label{bound}  l(f^m(\alpha_{-m})) &\ge  l(\alpha_0) - 2\sum_{i=0}^{m-1} l(f^{m-1-i}(\tau_{-m+i})) \\ 
&\ge l(\alpha_0) - 2\sum_{i=0}^{m-1} \lambda_f^{m-1-i} \rho \\
&= l(\alpha_0) - \frac{2\rho}{\lambda_f-1}(\lambda_f^m-1) \end{aligned}\]

There are three cases to consider:

\noindent {\bf Case 1, $(2,kn)$-Condition \ref{c2} holds for $f^{m-kn}(\alpha_{-m})$.} 

Since $(2,kn)$-Condition \ref{c2} holds for $f^{m-kn}(\alpha_{-m})$, by induction, we get the inequality:
\[\begin{aligned} l(\alpha_{-m}) 
&\ge 2^{2L+r} l(f^m(\alpha_{-m})) \\
&\ge 2^{2L+r} l(\alpha_0) - \frac{2\rho}{\lambda_f-1}(\lambda_f^m-1)2^{2L+r} \\
& > 8 l(\alpha_0) - H(\rho) \\
& \ge 7 l(\alpha_0)  & \text{ as } l(\alpha_0) \ge H(\rho)
\end{aligned}\]

\noindent {\bf Case 2, $\alpha$ is unidirectional and $(2,kn)$-Condition \ref{c1} holds for $f^{m-kn}(\alpha_{-m})$.} 

The top half of the annulus gives us:

\[  f^{2m}(\alpha_{-m}) \simeq f^{2m-1}(\tau_{-m})\cdots \tau_{m-1} \alpha_m \bar \tau_{m-1} \cdots f^{2m-1}(\bar \tau_{-m}) \]

Since $(2,kn)$-Condition \ref{c1} holds for  $f^{m-kn}(\alpha_{-m})$, by induction, we get the inequality:

\[\begin{aligned} 2^{2L+r} \left(l(\alpha_0) - \frac{2\rho}{\lambda_f-1}(\lambda_f^m-1) \right)
&\le 2^{2L+r} l(f^m(\alpha_{-m})) \\
&\le l(f^{2m}(\alpha_{-m})) \\
&\le l(\alpha_m) + 2\sum_{i=0}^{2m-1} l(f^{m-1-i}(\tau_{-m+i})) \\
&\le l(\alpha_m) + \frac{2\rho}{\lambda_f-1}(\lambda_f^{2m}-1)
\end{aligned}\]
Rearranging this inequality:
\[\begin{aligned} l(\alpha_m) &\ge 2^{2L+r} l(\alpha_0) - \frac{2\rho}{\lambda_f-1}(\lambda_f^m-1)2^{2L+r} - \frac{2\rho}{\lambda_f-1}(\lambda_f^{2m}-1) \\
& > 8 l(\alpha_0) - H(\rho) \\
& \ge 7 l(\alpha_0)
\end{aligned}\]

\noindent {\bf Case 3, $\alpha$ is bidirectional and $(2,kn)$-Condition \ref{c1} holds for $f^{m-kn}(\alpha_{-m})$.} 

By the hypothesis on strictly bidirectional annuli and inequality $m > 2L$, there exists $s,t \ge 1$  such that $\beta(0) < \beta(s) = \beta(m)-1 < \beta(s+t) = \beta(s+t+1)$. The top half of the annulus induces two homotopies:

\[\begin{aligned}f^{m+s+t}(\alpha_{-m}) &\simeq f^{m+s+t-1}(\tau_{-m})\cdots \tau_{s+t-1} \alpha_{s+t} \bar \tau_{s+t-1} \cdots f^{m+s+t-1}(\bar \tau_{-m}) \\ 
f^{t}(\alpha_{m}) &\simeq f^{t-1}(\bar\tau_{m-1})\cdots \bar\tau_{s+t} f(\alpha_{s+t})  \tau_{s+t} \cdots f^{t-1}(\tau_{m-1}) \end{aligned}\]
Choose $u \in [0,kn)_\mathbb{Z}$ so that $v = \frac{s+t+u}{kn} = \left \lceil \frac{s+t}{kn}\right \rceil$ is an integer. Since $(2,kn)$-Condition~ \ref{c1} holds for $f^{m-kn}(\alpha_{-m})$, we get \begin{equation}\label{eq1} 2^v l(f^m(\alpha_{-m})) \le l(f^{u}(\alpha_{s+t})) + \frac{2\rho}{\lambda_f-1}(\lambda_f^{2m}-1)\lambda_f^u  \end{equation}

The second homotopy gives us:\begin{equation}\label{eq2}\lambda_f^{t+u-1} l(\alpha_{m}) \ge l(f^{t+u-1}(\alpha_m)) \ge l(f^u(\alpha_{s+t})) - \frac{2\rho}{\lambda_f-1}(\lambda_f^t-1)\lambda_f^{u-1}  \end{equation}
Rewrite (\ref{eq2}) and combine it with (\ref{eq1}) to get 
\[ 2^v l(f^m(\alpha_{-m})) \le \lambda_f^{t+u-1} l(\alpha_m) + \frac{2\rho}{\lambda_f-1}(\lambda_f^t-1)\lambda_f^{u-1} + \frac{2\rho}{\lambda_f-1}(\lambda_f^{2m}-1)\lambda_f^{u} \]
\[\begin{aligned}\therefore \quad \frac{2^v}{\lambda_f^{t+u-1}} \left(l(\alpha_0) - \frac{2\rho}{\lambda_f-1}(\lambda_f^m-1) \right) &\le \frac{2^v}{\lambda_f^{t+u-1}} l(f^m(\alpha_{-m})) \\
&\le l(\alpha_{m}) + \frac{2\rho}{\lambda_f-1}\frac{\lambda_f^t-1}{\lambda_f^t} + \frac{2\rho}{\lambda_f-1}\frac{\lambda_f^{2m}-1}{\lambda_f^{t-1}} \\
&\le l(\alpha_m) + \frac{4\rho}{\lambda_f-1}(\lambda_f^{2m}-1) \qquad \text{ since } \lambda_f \ge 1\\
\end{aligned}\]
By construction and hypothesis, $m = (2L+r)kn = s+2t$, $u < kn$, and $t < L$. These can be used to show $v > L + r$ and $t+u-1 \le L(kn)$, so $ \frac{2^v}{\lambda_f^{t+u-1}} > 2^r\left(\frac{2}{\lambda_f^{kn}}\right)^L \ge 8$ from our choice of $r$. Therefore, 
\[\begin{aligned} l(\alpha_m) &\ge 8 l(\alpha_0) - \frac{2\rho}{\lambda_f-1}(\lambda_f^m-1)8 - \frac{4\rho}{\lambda_f-1}(\lambda_f^{2m}-1) \\
& > 8 l(\alpha_0) - H \\
&\ge 7 l(\alpha_0)
\end{aligned}\]

We have covered all the cases and shown that $M_f$ satisfies the annuli flare condition. By the combination theorem, $\pi_1(M_f)$ is word-hyperbolic.
\end{proof}

\begin{eg}[Continuing Example \ref{b2}]\label{b3} Since $\hat H_2$ is empty, the Sapir group is word-hyperbolic.\end{eg}

\begin{thm}\label{mainthm} Suppose $f:\Gamma \to \Gamma$ is an immersion. The following are equivalent:
 \begin{enumerate} 
 \item $\pi_1(M_f)$ is word-hyperbolic.
 \item $\pi_1(M_f)$ contains no $BS(1,d)$ subgroups for $d \ge 1$.
 \item $f$ has no invariant loop.
 \item $f$ is $(\lambda,n)$-hyperbolic and $\hat \Gamma_L = \emptyset$ for some $L \ge 1$.
 \end{enumerate} \end{thm}

\begin{proof} {~}

\noindent $(1){\implies}(2)$: $BS(1,d)$ subgroups are an obstruction to hyperbolicity.

\noindent $(2){\implies}(3)$: Due to Kapovich \cite[Lemma 2.3]{Kap}:

It is easy to see that having an $f$-invariant loop with degree $d$ implies there is a homomorphism $BS(1,d)\to\pi_1(M_f)$. Using normal forms, one shows that the homomorphism is injective.

\noindent $(3){\implies}(4)$: This is Lemma \ref{explem} and Proposition \ref{main}.

\noindent $(4){\implies}(1)$: This is Proposition \ref{main2} and Proposition \ref{keyprop}.\end{proof}

To remove the immersion assumption, it suffices to show that for any $\pi_1$-injective map
$f:\Gamma \to \Gamma$ with no invariant loop, the following holds:

\begin{enumerate}
\item $f$ is $(\lambda,n)$-hyperbolic --- this is analogous to Brinkmann's theorem/Lemma \ref{explem}. 
\item pullbacks stabilize --- this is analogous to Proposition \ref{main}.
\end{enumerate}

Combining Theorem \ref{mainthm} with Remark \ref{algo}, we get:

\begin{cor}\label{algo2} For any immersion $f$, there is a computable integer $L(f)$ such that $f$ has no invariant loop with degree $d \ge 2$ if and only if $\hat \Gamma_{L(f)}$ is empty. \end{cor}


\section{Fully Irreducible Endomorphisms}\label{irred}

In this section, $F_n$ is a free group with finite rank $n \ge 2$.

\begin{defn} An injective endomorphism $\phi: F_n \to F_n$ is {\bf fully irreducible} or { \bf irreducible with irreducible powers (iwip)} if there are no positive integer $k$, proper free factor $A \le F_n$, and element $x \in F_n$ such that $\phi^k(A) \le x A x^{-1}$. \end{defn}

The following is an unpublished theorem due to Patrick Reynolds. Reynolds' statement is stronger than the version given here but this weaker form suffices for our purposes.

\begin{thm}[{\cite[Corollary 5.5]{Rey}}] If $\phi:F_n \to F_n$ is nonsurjective and fully irreducible, then it is induced by an immersion $f:\Gamma \to \Gamma$. \end{thm}

\begin{cor}\label{cor2} Let $\phi:F_n\to F_n$ be a fully irreducible endomorphism. The following are equivalent:
\begin{enumerate}
\item $F_n*_\phi$ is word-hyperbolic. 
\item $F_n*_\phi$ has no $BS(1,d)$ subgroups for $d \ge 1$.
\item There are no $k, d \ge 1$ and nontrivial $g \in F_n$ such that $[\phi^k(g)] = [g^d]$.
\end{enumerate}\end{cor}
\begin{proof} We need to show $(3){\implies}(1)$: If $\phi$ is surjective, then this is Brinkmann's theorem. So we may assume $\phi$ is nonsurjective. By Reynolds' theorem, $\phi$ is induced by an immersion $f:\Gamma \to \Gamma$. The implication now follows from Theorem \ref{mainthm}.
\end{proof}

The following related corollary was suggested by Ilya Kapovich.

\begin{cor}\label{cor3} Let $\phi:F_2\to F_2$ be an injective endomorphism. The following are equivalent:
\begin{enumerate}
\item $F_2*_\phi$ is word-hyperbolic. 
\item $F_2*_\phi$ has no $BS(1,d)$ subgroups for $d \ge 1$.
\item There are no $k, d \ge 1$ and nontrivial $g \in F_n$ such that $[\phi^k(g)] = [g^d]$.
\end{enumerate}\end{cor}

\begin{proof} We need to show $(3){\implies}(1)$: In $F_2 = \langle a, b \rangle$, automorphisms either fix the conjugacy class of $[a,b]=aba^{-1}b^{-1}$ or map it to the class of its inverse  $bab^{-1}a^{-1}$. Since $\phi$ is atoroidal, it must be nonsurjective. Proper free factors of $F_2$ are infinite cyclic; as no power of $\phi$ maps a nontrivial element to a conjugate of its power, then no power of $\phi$ maps a proper free factor into a conjugate of itself. Thus $\phi$ is fully irreducible.
By Reynolds' theorem, $\phi$ is induced by an immersion $f:\Gamma \to \Gamma$. The lack of $f$-invariant loops implies $\pi_1(M_f) \cong F_2*_\phi$ is word-hyperbolic by Theorem \ref{mainthm}.\end{proof}

We will now give some examples. For $n \in \{2,3,4\}$, let $\varphi_n: F_n \to F_n$ be defined by:
\[  \varphi_2: \begin{aligned} a &\mapsto b a^{-1} \\ b &\mapsto a^2(a b^{-1})^2 a^2 \end{aligned} 
\qquad \varphi_3: \begin{aligned} a &\mapsto b \\ b  &\mapsto a c a^{-1} \\ c &\mapsto a b^{-1}ac^{-1}a^{-1}ba \end{aligned}
 \qquad \varphi_4: \begin{aligned} a &\mapsto b \\ b  &\mapsto a c \\ c &\mapsto d a^{-1} \\ d &\mapsto ab^{-1}ad^{-1}b^{-1}aca \end{aligned} \]
 
$G = F_2 *_{\varphi_2}$ is a 1-relator group and we used Brown's algorithm \cite{Bro} to compute the BNS invariant, a cone $\Sigma(G) \subset H^1(G; \mathbb R)$. The two other endomorphisms above were found by choosing two rational rays in $\Sigma(G)$ and rewriting the presentation of $G$ as an asc. HNN-ext. with respect to these rays. So $G \cong F_3 *_{\varphi_3} \cong F_4*_{\varphi_4}$. This group's BNS invariant has no symmetric subset, so we can conclude that $G$ never splits as a free-by-cyclic group \cite{BNS}. Consequently, Brinkmann's result (and Kapovich's as far as we can tell) can not be used to prove word-hyperbolicity of $G$.

\begin{claim*} No iterate of $\varphi_2, \varphi_3$, or $\varphi_4$ maps a nontrivial element to a conjugate of its power and $G$~is word-hyperbolic.\end{claim*}

\begin{proof}[Sketch proof]By the preceding discussion, these endomorphisms were constructed so that their HNN-extensions are isomorphic to the same 1-relator group $G$. The endomorphism $\varphi_3$ is induced by an {\it irreducible} immersion $f_3$ (See second graph in Figure~\ref{fig4}). This immersion has {\it Perron-Frobenius} eigenvalue $\lambda_3 = 1 + \sqrt{3}$ which is not a root of any integer. As, $f_3$ expands all immersed loops uniformly by $\lambda_3$, it has no invariant loop. By Corollary \ref{cor2}, $G \cong F_3 *_{\varphi_3}$ is word-hyperbolic. But $G =  F_2 *_{\varphi_2} \cong  F_4 *_{\varphi_4}$. Applying Corollary \ref{cor2} again, no iterate of $\varphi_2$ and $\varphi_4$ maps a nontrivial element to a conjugate of its power.
\end{proof}

As argued in the proof of Corollary \ref{cor3}, in $F_2$, the lack of $\varphi_2^k$-invariant  cyclic subgroups, for all $k \ge 1$ implies the endomorphism $\varphi_2$ is fully irreducible and induced by an immersion. We found immersions inducing both $\varphi_2$ and $\varphi_4$ (Figure~\ref{fig4}).

An interesting area of further study is characterizing all the nonsurjective injective endomorphisms that are not covered by the main theorem; more precisely, which nonsurjective endomorphisms are not induced by an immersion. For $F_2$, the answer seems to be only endomorphisms that take the following form up to change of basis:
\[  \psi: \begin{aligned} a &\mapsto a \\ b &\mapsto a^k b\ldots b \end{aligned} 
\qquad  \text{or} \qquad \psi: \begin{aligned} a &\mapsto a^{-1} \\ b  &\mapsto (a^k b\ldots b)^{-1} \end{aligned} 
\qquad \qquad (k\ge 1)\]

\begin{figure}
 \centering 
 \includegraphics[scale=0.6]{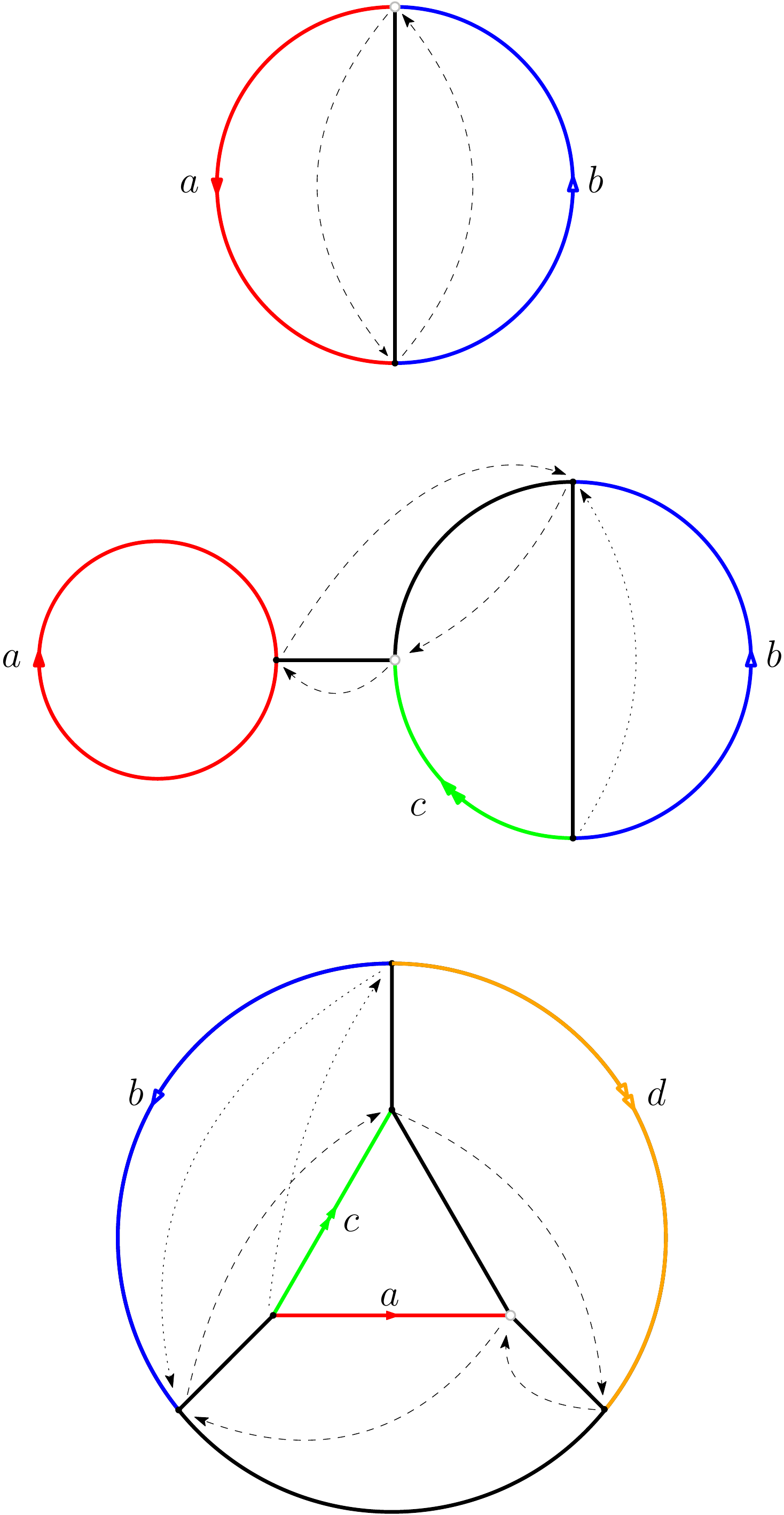}
 \caption{Markings for the graph immersions inducing $\varphi_i$.}
The dashed/dotted arrows shows the orbits of the vertices under the immersions. The white vertex is the marked point. The marking, vertex orbits, and endomorphisms' definitions are enough to (re)construct the immersions.
 \label{fig4}
\end{figure}

\newpage

\bibliography{refs}

\begin{thebibliography}{10}

\bibitem{ABC}
J.~M. Alonso, T.~Brady, D.~Cooper, V.~Ferlini, M.~Lustig, M.~Mihalik,
  M.~Shapiro, and H.~Short.
\newblock Notes on word hyperbolic groups.
\newblock In {\em Group theory from a geometrical viewpoint ({T}rieste, 1990)},
  pages 3--63. World Sci. Publ., River Edge, NJ, 1991.
\newblock Edited by Short.

\bibitem{BF92}
Mladen Bestvina and Mark Feighn.
\newblock A combination theorem for negatively curved groups.
\newblock {\em J. Differential Geom.}, 35(1):85--101, 1992.

\bibitem{BNS}
Robert Bieri, Walter~D. Neumann, and Ralph Strebel.
\newblock A geometric invariant of discrete groups.
\newblock {\em Invent. Math.}, 90(3):451--477, 1987.

\bibitem{Bri}
Peter Brinkmann.
\newblock Hyperbolic automorphisms of free groups.
\newblock {\em Geom. Funct. Anal.}, 10(5):1071--1089, 2000.

\bibitem{Bro}
Kenneth~S. Brown.
\newblock Trees, valuations, and the {B}ieri-{N}eumann-{S}trebel invariant.
\newblock {\em Invent. Math.}, 90(3):479--504, 1987.

\bibitem{But}
J.~O. {Button}.
\newblock {Strictly ascending HNN extensions of finite rank free groups that
  are linear over Z}.
\newblock {\em ArXiv e-prints}, February 2013.
\newblock \href{https://arxiv.org/abs/1302.5370}{arxiv:1302.5370}.

\bibitem{Fried}
Joel Friedman.
\newblock Sheaves on graphs, their homological invariants, and a proof of the
  {H}anna {N}eumann conjecture: with an appendix by {W}arren {D}icks.
\newblock {\em Mem. Amer. Math. Soc.}, 233(1100):xii+106, 2015.

\bibitem{Gau}
Fran\c{c}ois Gautero.
\newblock Hyperbolicity of mapping-torus groups and spaces.
\newblock {\em Enseign. Math. (2)}, 49(3-4):263--305, 2003.

\bibitem{Gro}
M.~Gromov.
\newblock Hyperbolic groups.
\newblock In {\em Essays in group theory}, volume~8 of {\em Math. Sci. Res.
  Inst. Publ.}, pages 75--263. Springer, New York, 1987.

\bibitem{Kap}
Ilya Kapovich.
\newblock Mapping tori of endomorphisms of free groups.
\newblock {\em Comm. Algebra}, 28(6):2895--2917, 2000.

\bibitem{Min}
Igor Mineyev.
\newblock Submultiplicativity and the {H}anna {N}eumann conjecture.
\newblock {\em Ann. of Math. (2)}, 175(1):393--414, 2012.

\bibitem{HN}
Hanna Neumann.
\newblock On the intersection of finitely generated free groups. {A}ddendum.
\newblock {\em Publ. Math. Debrecen}, 5:128, 1957.

\bibitem{WN}
Walter~D. Neumann.
\newblock On intersections of finitely generated subgroups of free groups.
\newblock In {\em Groups---{C}anberra 1989}, volume 1456 of {\em Lecture Notes
  in Math.}, pages 161--170. Springer, Berlin, 1990.

\bibitem{Rey}
Patrick {Reynolds}.
\newblock {Dynamics of Irreducible Endomorphisms of $F_n$}.
\newblock {\em ArXiv e-prints}, August 2010.
\newblock \href{https://arxiv.org/abs/1008.3659}{arxiv:1008.3659}.

\bibitem{ThII}
William~P. Thurston.
\newblock Three-dimensional manifolds, {K}leinian groups and hyperbolic
  geometry.
\newblock {\em Bull. Amer. Math. Soc. (N.S.)}, 6(3):357--381, 1982.

\end{thebibliography}
\bibliographystyle{plain}

\end{document}